\documentclass[reqno]{amsart}
\usepackage{amsfonts}
\usepackage{amssymb}
\usepackage{times}
\usepackage{xcolor}
\usepackage{soul}

\usepackage{scalerel}
\usepackage{stackengine,wasysym}

\newcommand\reallywidetilde[1]{\ThisStyle{%
  \setbox0=\hbox{$\SavedStyle#1$}%
  \stackengine{-.1\LMpt}{$\SavedStyle#1$}{%
    \stretchto{\scaleto{\SavedStyle\mkern.2mu\AC}{.5150\wd0}}{.6\ht0}%
  }{O}{c}{F}{T}{S}%
}}

\def\g2{ \hbox{\got g}_2}
\def\f4{\hbox{\got f}_4}
\def\e6{\hbox{\got e}_6}
\def\fs4{\hbox{\gots f}_4}
\def\F4{\hbox{\got F}_4}
\def\Fs4{\hbox{\gots F}_4}

\def\al{\ifcase\xypolynode\or F \or A\or B\or C\or D\or G\fi}
\def\ala{\ifcase\xypolynode\or a \or b\or c\or d\or g\or f\fi}

\makeatletter
\def\Ddots{\mathinner{\mkern1mu\raise\p@
\vbox{\kern7\p@\hbox{.}}\mkern2mu
\raise4\p@\hbox{.}\mkern2mu\raise7\p@\hbox{.}\mkern1mu}}
\makeatother

{\theoremstyle{plain}%
  \newtheorem{theorem}{Theorem}[section]
  \newtheorem{corollary}{Corollary}[section]
  \newtheorem{proposition}{Proposition}[section]
  \newtheorem{lemma}{Lemma}[section]

  \newtheorem{definition}{Definition}[section]
  \newtheorem{notation}[definition]{Notation}}

\newtheorem{remark}{Remark}[section]

\newfont{\hueca}{msbm10}
\def\hu #1{\hbox{\hueca #1}}\def\hu #1{\hbox{\hueca #1}}
\begin{document}

\title{Leibniz superalgebras with a set grading}

\thanks{The first and second  authors acknowledge financial assistance by the Centre for Mathematics of the University of Coimbra (UID/MAT/00324/2019, funded by the Portuguese
Government through FCT/MEC and co-funded by the European Regional Development Fund through the Partnership Agreement PT2020). Third and fourth authors are supported by the PCI of the UCA `Teor\'\i a de Lie y Teor\'\i a de Espacios de Banach', by the PAI with project number FQM298.}

\author[H. Albuquerque ]{Helena~Albuquerque}

\address{Helena~Albuquerque, CMUC, Departamento de Matem\'atica, Universidade de Coimbra, Apartado 3008,
3001-454 Coimbra, Portugal. \hspace{0.1cm} {\em E-mail address}: {\tt lena@mat.uc.pt}}

\author[E. Barreiro]{Elisabete~Barreiro}

\address{Elisabete~Barreiro, CMUC, Departamento de Matem\'atica, Universidade de Coimbra, Apartado 3008,
3001-454 Coimbra, Portugal. \hspace{0.1cm} {\em E-mail address}: {\tt mefb@mat.uc.pt}}{}

\author[A.J. Calder\'on]{A.J.~Calder\'on}

\address{A.J.~Calder\'on, Departamento de Matem\'aticas, Universidad de C\'adiz, Campus de Puerto Real, 11510, Puerto Real, C\'adiz, Espa\~na. \hspace{0.1cm} {\em E-mail address}: {\tt ajesus.calderon@uca.es}}{}

\author[Jos\'{e} M. S\'{a}nchez]{Jos\'{e} M. S\'{a}nchez}

\address{Jos\'{e} M. S\'{a}nchez, Departamento de Matem\'aticas, Universidad de C\'adiz, Campus de Puerto Real, 11510, Puerto Real, C\'adiz, Espa\~na. \hspace{0.1cm} {\em E-mail address}: {\tt txema.sanchez@uca.es}}

\begin{abstract}
Consider a Leibniz superalgebra  $\mathfrak L$ additionally graded by an arbitrary set $I$ (set grading). We show that $\mathfrak L$ decomposes as the sum of well-described graded ideals plus (maybe) a suitable linear subspace.  In the case of ${\mathfrak L}$ being of maximal length, the simplicity of ${\mathfrak L}$ is also characterized in terms of connections.



\end{abstract}

\maketitle

\section{Introduction and previous definitions}

One one hand, Leibniz superalgebras arise as an extension of the Leibniz algebras (see \cite{8,5}, among others) in a similar way that Lie superalgebras extends to Lie algebras. Indeed, the class of Leibniz superalgebras also extends the one of Lie superalgebras by removing the skew-supersymmetry, which is of interest in the formalism of mechanics of Nambu \cite{Dale}. On the other hand, the interest in gradings on superalgebras has been remarkable in the last years (for example see \cite{Draper2, Sudarkin} for the class of Lie superalgebras and \cite{Kac} for the case of the Jordan superalgebra $K_{10}$). However gradings by means of an arbitrary set, not necessarily a group, have been considered in  the literature just in a slightly way. This kind of graduation was presented in \cite{19} and called Lie gradings. A complete and recent review of the state of the art can be found in \cite{Origen2}. Motivated by the results obtained for Lie and Leibniz algebras in \cite{Origen1, Leibniz_set_graded}, respectively, in the present paper we study arbitrary Leibniz superalgebras (not necessarily simple or finite-dimensional) and over an arbitrary base field $\mathbb{K}$ graded by means of an arbitrary set $I$, by focusing on its structure. In Section 2 we extend the techniques of connections in the support of the set-grading to the framework of a Leibniz superalgebra so as to show that it is of the form ${\frak L} = U + \sum_jI_j$ with $U$ a linear subspace of a distinguish homogeneous subspace ${\frak L}_{\mathfrak{o}}$ and any $I_j$ a well described set-graded ideal of ${\frak L}$, satisfying $[I_j,I_k] = 0$ if $j \neq k$. In the final section, and under certain conditions, we focuss on those of maximal length and characterize the simplicity of this class of superalgebras in terms of connections.

\begin{definition}\rm
A {\it Leibniz superalgebra} $({\frak L}, [\cdot, \cdot])$ is a ${\hu Z}_2$-graded algebra ${\frak L} = {\frak L}^{\bar 0} \oplus {\frak L}^{\bar 1}$  over an arbitrary base field ${\mathbb K}$  endowed with a bilinear product $[\cdot, \cdot]$ satisfying
\begin{itemize}
\item[i.] $[x,y] \in {\frak L}^{{\bar i}+{\bar j}}$
\item[ii.] $[x,[y,z]] = [[x,y],z] - (-1)^{{\bar j}{\bar k}}[[x,z],y]$ \hspace{0.2cm} {\em (Super Leibniz identity)}
\end{itemize}
for any homogenous elements $x \in {\frak L}^{\bar i}, y \in {\frak L}^{\bar j}, z \in {\frak L}^{\bar k},$ with ${\bar i}, {\bar j}, {\bar k} \in {\hu Z}_2$.
\end{definition}

\begin{definition}\rm
For an arbitrary (non-empty) set $I,$ we say that a Leibniz superalgebra  ${\frak L}$   {\it has a set grading} (by means of $I$) or it is {\it set-graded} if $${\frak L} = \bigoplus\limits_{a \in I} {\frak L}_a,$$  for any $a, b \in I$; and either $[{\frak L}_a, {\frak L}_b] = 0$ or $0 \neq [{\frak L}_a, {\frak L}_b] \subset {\frak L}_c$ for some (unique) $c \in I$, and where the homogeneous components are (graded) linear subspace satisfying
\begin{equation}\label{separa}
\hbox{${\frak L}_a = {\frak L}_a^{\bar 0} \oplus {\frak L}_a^{\bar 1},$ \hspace{0.1cm} being ${\frak L}_a^{\bar i} := {\frak L}_a \cap {\frak L}^{\bar i},$ \hspace{0.1cm} for ${\bar i} \in {\hu Z}_2$.}
\end{equation}

\end{definition}
\noindent Clearly ${\frak L}^{\overline 0}$ is a set-graded Leibniz algebra (see \cite{Leibniz_set_graded}). Moreover, if the identity $[x,y]=-(-1)^{{\bar i}{\bar j}}[y,x]$, with $  {\bar i}, {\bar j} \in {\hu Z}_2$, holds then Super Leibniz identity becomes Super Jacobi identity and so set-graded Leibniz superalgebras generalize set-graded Lie superalgebras (therefore, also set-graded Lie algebras studied in \cite{Origen1}), which is of interest in the forma\-lism of mechanics of Nambu \cite{Dale}. We note that a set grading of ${\frak L}$ provides a refinement of the initial ${\hu Z}_2$-grading of ${\frak L}$ and that split Leibniz superalgebras, graded Leibniz superalgebras, split Leibniz algebras, graded Leibniz algebras, split Lie superalgebras, graded Lie superalgebras, split Lie algebras and graded Lie algebras are examples of set-graded Leibniz superalgebras. Hence, the present paper also extends the results in \cite{graalg}-\cite{Yo2}.

 The usual regularity concepts are considered in a graded sense (compatible with the initial ${\hu Z}_2$-graduation of ${\frak L}$). A  $\mathbb{Z}_2$-graded subspace $A = A^{\bar  0} \oplus A^{\bar  1}$ of ${\frak L}$  is called a {\it subsuperalgebra}  of ${\frak L}$ if $[A,A] \subset A$.  A {\it (graded) ideal} $\mathcal{I}$ of ${\frak L}$ is a subsuperalgebra of ${\frak L}$ such that $[\mathcal{I},{\frak L}] + [{\frak L},\mathcal{I}] \subset \mathcal{I}.$ A (graded) ideal $\mathcal{I}$ of ${\frak L}$ splits as
\begin{equation}\label{idealpartio}
\hbox{$\mathcal{I} = \bigoplus\limits_{a \in I}\mathcal{I}_a = \bigoplus\limits_{a \in I}(\mathcal{I}_a^{\bar 0} \oplus \mathcal{I}_a^{\bar 1})$ with any $\mathcal{I}_a^{\bar i} := \mathcal{I}_a \cap {\frak L}^{\bar i}, {\bar i} \in {\hu Z}_2$.}
\end{equation}

The (graded) ideal ${\frak I}$ generated by
\begin{equation}\label{equi3}
\Bigl\{ [x,y] + (-1)^{{\bar i}{\bar j}}[y,x] : x \in {\frak L}^{\bar i}, y \in {\frak L}^{\bar j}, \bar i, \bar j \in {\hu Z}_2 \Bigr\}
\end{equation}
plays an important role in the theory since it determines the (possible) non-super Lie cha\-racter of ${\frak L}$. From definition of ideal $[{\frak I}, {\frak L}] \subset {\frak I}$ and from Super Leibniz identity, it is straightforward to check that this ideal satisfies
\begin{equation}\label{equi}
[{\frak L},{\frak I}] = 0.
\end{equation}

 Here we note that the usual definition of simple superalgebra lacks of interest in the case of Leibniz superalgebras because would imply the ideal ${\frak I}={\frak L}$ or ${\frak I}={0}$, being so ${\frak L}$ an abelian (product zero)  or a Lie superalgebra, respectively (see Equation (\ref{equi})). Abdykassymova and Dzhumadil'daev introduced in \cite{Abdy, Dzhu} an  adequate definition in the case of Leibniz algebras $(L,[\cdot, \cdot])$ by calling simple to the ones such that its only ideals are $\{0\}$, $L$ and the one generated by the set $\{[x,x]: x \in L\}$.  Following this vain, we consider the next definition.

\begin{definition}\label{Defsimple}\rm
A set-graded Leibniz superalgebra ${\frak L}$ is called {\it simple} if $[{\frak L},{\frak L}] \neq 0$ and its only (graded) ideals are $\{0\}, {\frak I}$ and ${\frak L}$.
\end{definition}

\noindent Observe that from the grading of ${\frak L}$ and Equation
\eqref{separa} we get, for any $a,b \in I,$ and  ${\bar i}, {\bar j} \in {\hu Z}_2$, $$[{\frak L}_a^{\bar i}, {\frak L}_b^{\bar j}] \subset {\frak L}_c^{{\bar i}+{\bar j}}$$  with some (unique) $c \in I$. We call the {\it support} of the set grading to the set $$\frak{S} := \bigl\{ a \in I :  {\frak L}_a \neq 0 \bigr\}.$$ We also denote by $\frak{S}^{\bar 0} := \bigl\{ a \in I : {\frak L}_a^{\bar 0} \neq 0 \bigr\}$ and by $\frak{S}^{\bar 1} := \bigl\{ a \in I : {\frak L}_a^{\bar 1} \neq 0 \bigr\}$. So $\frak{S} = \frak{S}^{\bar 0} \cup \frak{S}^{\bar 1}$,
being a non necessarily disjoint union.


\section{Connections in the support. Decompositions}

We begin this section by developing, as the main tool, connections techniques in the support of a set-graded Leibniz superalgebras. In the paper, ${\frak L} = \bigoplus_{i \in \frak{S}}({\frak L}_i^{\bar 0} \oplus {\frak L}_i^{\bar 1})$ is an arbitrary set-graded Leibniz superalgebra and the set $\frak S$ the support of the set grading. For each $a \in \frak{S}$, a new variable $\tilde{a} \notin  \frak{S}$ is introduced and we denote by
$$\tilde{\frak{S}} := \bigl\{ \tilde{a} : a \in \frak{S} \bigr\}$$
the set consisting of all these new symbols. For any $ \tilde{a}
\in \tilde{\frak{S}}$,  we denote
$$\widetilde{(\tilde{a})}:=a.$$
We will denote by $\mathcal{P}(A)$ the power set of a given set $A.$ Next, we consider the following operation, $$\star : (\frak{S} \dot{\cup} \tilde{\frak{S}}) \times (\frak{S} \dot{\cup} \tilde{\frak{S}}) \to \mathcal{P}(\frak S)$$ such as:

\begin{itemize}
\item for $a,b \in \frak{S},$

$$\begin{array}{l}
a \star b :=\left\{
\begin{array}{cll}
\emptyset &\text{if}& [{\frak L}_a,{\frak L}_b] = 0,\\
\{c\} &\text{if} & 0 \neq [{\frak L}_a, {\frak L}_b] \subset {\frak L}_c,
\end{array}
\right.
\end{array}$$

\item for $a \in \frak{S}$ and $\tilde{b} \in \tilde{\frak{S}},$

$$a \star \tilde{b} = \tilde{b} \star a := \{c \in \frak{S} : 0 \neq [{\frak L}_c, {\frak L}_b] \subset {\frak L}_a\},$$

\item for $\tilde{a}, \tilde{b} \in \tilde{\frak{S}},$

$$\tilde{a} \star \tilde{b} := \emptyset.$$
\end{itemize}

Given any subset $\emptyset \neq \mathcal{U}$ of $\frak{S} \cup \tilde{\frak{S}},$ we write $\tilde{\mathcal{U}} := \{\tilde{a} : a \in \mathcal{U}\}$, and also $\tilde{\emptyset} := \emptyset$.

At this moment we have to note that sometimes it is interesting to distinguish one element $\mathfrak{o}$ in the support of the grading, because the homogeneous space ${\frak L}_{\mathfrak{o}}$ has, in a sense, a special behavior to the remaining elements in the set of homogeneous spaces ${\frak L}_a$, for $a \in \frak{S}$. This is for instance the case in which the grading set $I$ is an abelian group, where the homogeneous space $\frak{L}_0$ associated to the unit element $0$ in the group enjoys a distinguished role. Indeed, if we consider the group grading determined by the Cartan decomposition of a semisimple finite-dimensional Lie algebra, the homogeneous space associated to the unit element agrees with the Cartan subalgebra $H$. That is ${\frak L}_0 = H$, being then $\dim({\frak L}_g) = 1$ for any $g$ in the support of the grading up to $\dim({\frak L}_0)$ which is not bounded by this condition. The same phenomenon happens for the more general case of locally finite split Lie algebras and, in general, for group-graded Lie algebras of maximal length (see \cite{graalg, Stumme}). From here, we are going to feel free in our study to distinguish one special element $\mathfrak{o}$ in the support of the grading. Hence, let us now fix an element $\mathfrak{o}$ such that either $\mathfrak{o} \in \frak{S}$ satisfying the property $\mathfrak{o} \star a \neq \{\mathfrak{o}\}$ for any $a \in \frak{S} \setminus \{\mathfrak{o}\}$, or $\mathfrak{o} = \emptyset$. Note that the possibility $\mathfrak{o} = \emptyset$ holds for the case in which it is not wished to distinguish any element in $\frak S$.
Finally, we need to introduce the following mapping:

$\phi: \mathcal{P}((\frak{S} \dot{\cup} \tilde{\frak{S}}) \setminus \{\mathfrak{o},\tilde{\mathfrak{o}}\}) \times (\frak{S} \dot{\cup} \tilde{\frak S}) \to \mathcal{P}((\frak{S} \dot{\cup} \tilde{\frak{S}}) \setminus \{\mathfrak{o},\tilde{\mathfrak{o}}\})$ defined as
\begin{itemize}
\item $\phi(\emptyset, \frak{S} \dot{\cup} \tilde{\frak{S}}) = \emptyset,$
\item for any $\emptyset \neq \mathcal{U} \in \mathcal{P}((\frak{S} \dot{\cup} \tilde{\frak{S}}) \setminus \{\mathfrak{o},\tilde{\mathfrak{o}}\})$ and $r \in \frak{S} \dot{\cup} \tilde{\frak{S}},$   $$\phi(\mathcal{U},r) := \Bigl( \Bigl(\bigcup_{x \in \mathcal{U}}(x \star r)\Bigr) \setminus       \{ \mathfrak{o} \} \Bigr) \cup \reallywidetilde{ \Bigl(\Bigl( \bigcup_{x \in \mathcal{U}} (x \star r) \Bigr) \setminus \{ \mathfrak{o} \}\Bigr)}.$$
\end{itemize}

\noindent Note that for any $\mathcal{U} \in \mathcal{P}((\frak{S} \dot{\cup} \tilde{\frak{S}}) \setminus \{\mathfrak{o},\tilde{\mathfrak{o}}\})$ and $r \in \frak{S} \dot{\cup} \tilde{\frak{S}}$ we get that
\begin{equation}\label{igualdad}
\phi(\mathcal{U},r) = \widetilde{\phi(\mathcal{U},r)}
\end{equation}
and $$\phi(\mathcal{U},r) \cap \frak{S} = \Bigl( \bigcup_{x \in \mathcal{U}} (x \star r) \Bigr) \setminus \{\mathfrak{o}\}.$$
Also observe that for any $a \in \frak{S}$ and $r \in \frak{S} \dot{\cup} \tilde{\frak{S}}$ we have that:
\begin{itemize}
\item $a \in x \star r$ for some $x \in \frak{S}$ if and only if $x \in a \star \tilde{r},$
\item while $a \in c \star r$ for some $c \in \tilde{\frak{S}}$ if and only if $\tilde{c} \in \tilde{a} \star r$.
\end{itemize}
 These facts together with Equation \eqref{igualdad} imply that for any $\mathcal{U} \in \mathcal{P}((\frak{S} \dot{\cup} \tilde{\frak{S}}) \setminus \{\mathfrak{o},\tilde{\mathfrak{o}}\})$ such that $\mathcal{U} = \tilde{\mathcal{U}}$ and $r \in \frak{S} \dot{\cup} \tilde{\frak{S}}$ we have

$$\hbox{$a \in \phi(\mathcal{U},r) \cap \frak{S}$ if and only if $a \in \frak{S}$ and}$$
\begin{equation}\label{Ecuacion2}
\hbox{either $\phi(\{a\},\tilde{r}) \cap \mathcal{U} \cap \frak{S} \neq \emptyset$ or $\phi(\{\tilde{a}\},r) \cap \mathcal{U} \cap \frak{S} \neq \emptyset.$}
\end{equation}

\begin{definition}\label{con}\rm
Let $a,b \in \frak{S} \setminus \{\mathfrak{o}\}.$ We say that
$a$ is {\em connected} to $b$ if there exists $$\{r_1,r_2,\dots,r_n\} \subset \frak{S} \dot{\cup} \tilde{\frak{S}}$$ such that\\
If $n=1:$
\begin{enumerate}
\item[{\rm 1.}] $r_1 = a = b.$
\end{enumerate}
If $n \geq  2:$
\begin{enumerate}
\item[{\rm 1.}] $r_1 \in \{a,\tilde{a}\}.$
\item[{\rm 2.}] $\phi(\{r_1\}, r_2) \neq \emptyset,$\\
$\phi(\phi(\{r_1\}, r_2),r_3) \neq \emptyset,$\\
$\hspace*{1.6cm} \vdots$\\
$\phi(\phi(\dots \phi(\{r_1\}, r_2),\dots ,r_{n-2}),r_{n-1}) \neq \emptyset.$\\
\item[{\rm 3.}] $b \in \phi(\phi(\dots \phi(\{r_1\}, r_2),\dots ,r_{n-1}),r_n).$\\
\end{enumerate}
We say that $\{r_1,r_2,\dots,r_n\}$ is a {\em connection} from $a$ to $b$.
\end{definition}

\noindent The proof of the next result is analogous to the proof of \cite[Proposition 2.1]{Origen1}.

\begin{proposition}\label{pro1}
The relation $\sim$ in $\frak{S} \setminus \{\mathfrak{o}\}$, defined by $a \sim b$ if and only if $a$ is connected to $b$, is of equivalence.
\end{proposition}




Given $a \in \frak{S} \setminus \{\mathfrak{o}\}$, we denote by
$$[a] := \Bigl\{b \in \frak{S} \setminus \{\mathfrak{o}\} : b \hspace{0.1cm} {\rm is} \hspace{0.1cm} {\rm connected} \hspace{0.1cm} {\rm to} \hspace{0.1cm} a \Bigr\}.$$
\noindent By Proposition \ref{pro1} if $t \notin [a]$ then $[a] \cap [t] = \emptyset$.
Our next goal in this section is to associate an (adequate) graded ideal ${\frak L}_{[a]}$ to any $[a]$, with $a \in \frak{S} \setminus \{\mathfrak{o}\}$. For $[a]$, we define the set
$${\frak L}_{[a],\mathfrak{o}} := span_{\hu K} \Bigl\{[{\frak L}_b,{\frak L}_c] : b,c \in [a]\Bigr\} \cap {\frak L}_{\mathfrak{o}} =$$
\begin{equation}\label{suma16}
\Bigl(\sum\limits_{b,c \in [a]}([{\frak L}_b^{\bar 0},{\frak
L}_c^{\bar 0}] + [{\frak L}_b^{\bar 1},{\frak L}_c^{\bar 1}]) \oplus \sum\limits_{b,c \in [a]} ([{\frak L}_b^{\bar 0},{\frak L}_c^{\bar 1}] + [{\frak L}_b^{\bar 1},{\frak L}_c^{\bar 0}])\Bigr) \cap {\frak L}_{\mathfrak{o}}
\end{equation}
$$ \subset {\frak L}_{\mathfrak{o}}^{\bar 0} \oplus {\frak L}_{\mathfrak{o}}^{\bar 1},$$ last equality being consequence of Equation \eqref{separa}, where ${\frak L}_{\mathfrak{o}} := \{0\}$ whence $\mathfrak{o} = \emptyset.$ We also consider $$V_{[a]} := \bigoplus\limits_{b \in [a]}{\frak L}_b = \bigoplus\limits_{b \in [a]}\Bigl({\frak L}_b^{\bar 0} \oplus {\frak L}_b^{\bar 1}\Bigr).$$ Finally, we denote by ${\frak L}_{[a]}$ the following (graded) subspace of ${\frak L}$,
$${\frak L}_{[a]} := {\frak L}_{[a],\mathfrak{o}} \oplus V_{[a]}.$$

\begin{proposition}\label{pro_ideal}
For any $a \in \frak{S} \setminus \{\mathfrak{o}\}$, the (linear) subspace ${\frak L}_{[a]}$ is a Leibniz subsuperalgebra of ${\frak L}$.
\end{proposition}
\begin{proof}
We have to check that ${\frak L}_{[a]}$ satisfies $[{\frak L}_{[a]},{\frak L}_{[a]}] \subset {\frak L}_{[a]}$. Taking into account the expression of $\mathfrak{L}_{[a]}$ and the bilinearity of the product, we have
\begin{equation*}
[{\frak L}_{[a]}, {\frak L}_{[a]}] = [{\frak L}_{[a],\mathfrak{o}} \oplus V_{[a]}, {\frak L}_{[a],\mathfrak{o}} \oplus V_{[a]}]
\end{equation*}
\begin{equation}\label{cuatro2}
\subset [{\frak L}_{[a],\mathfrak{o}}, {\frak L}_{[a],\mathfrak{o}}] + [{\frak L}_{[a],\mathfrak{o}}, V_{[a]}] + [V_{[a]}, {\frak L}_{[a],\mathfrak{o}}]
+ [V_{[a]}, V_{[a]}].
\end{equation}

Consider the above second summand $[{\frak L}_{[a],\frak{o}}, V_{[a]}]$. Taking into account ${\frak L}_{[a],\frak{o}} \subset {\frak L}_{\frak{o}}$ and $[{\frak L}_{\frak o},{\frak L}_a] \subset {\frak L}_b$ for $\frak{o} \star a = b \in \frak{S} \setminus \{\mathfrak{o}\}$, the connection $\{b,\tilde{\frak{o}}\}$ gives $b \sim a$ and we have $[{\frak L}_{[a],{\frak o}}, V_{[a]}] \subset V_{[a]}$. In a similar way $[V_{[a]},{\frak L}_{[a],\frak{o}}] \subset V_{[a]}$ and so
\begin{equation}\label{diez2}
[{\frak L}_{[a],\frak{o}}, \oplus V_{[a]}] + [V_{[a]},{\frak L}_{[a],\frak{o}}] \subset V_{[a]}.
\end{equation}

Let us consider now the fourth summand $[V_{[a]}, V_{[a]}]$ in Equation \eqref{cuatro2}. Suppose there exist $b,c \in [a]$ such that $[{\frak L}_b,{\frak L}_c] \neq 0$. So $b \star c =d \in \frak{S}.$ If $d = \frak{o}$ then  clearly $[{\frak L}_b,{\frak L}_c]  \subset {\frak L}_\frak{o}$ and from Equation \eqref{suma16} we conclude  $[{\frak L}_b,{\frak L}_c] \subset {\frak L}_{[a],\frak{o}} $. Otherwise, if $d \in  \frak{S} \setminus \{\frak{o}\}$, using the connection $\{d, \tilde{c}\}$ then $b$ is connected to $d$  and we conclude $d \in [a]$. Hence, $[{\frak L}_b,{\frak L}_c] \subset {\frak L}_{d} \subset V_{[a]}$. In conclusion
\begin{equation}\label{nueve2}
[V_{[a]}, V_{[a]}] \subset {\frak L}_{[a]}.
\end{equation}

Finally, we consider the first summand $[{\frak L}_{[a],\frak{o}},{\frak L}_{[a],\frak{o}}]$ in \eqref{cuatro2}. Suppose now there exist $b,c, b',c' \in [a]$ with $b \star c = \frak{o}$ and $b' \star c' = \frak{o}$ such that $\Bigl[[{\frak L}_b, {\frak L}_c], [{\frak L}_{b'},{\frak L}_{c'}] \Bigr] \neq 0.$
Requiring to the expressions with $\mathbb{Z}_2$-graduation we have
$$\Bigl[ \sum\limits_{b,c \in [a]}[{\frak L}_b, {\frak L}_c],\sum\limits_{b',c' \in [a]} [{\frak L}_{b'},{\frak L}_{c'}] \Bigr] \subset \sum\limits_{{\tiny \begin{array}{c}
b,c ,b',c' \in [a]\\
\bar{i},\bar{j},\bar{k},\bar{l} \in \mathbb{Z}_2 \end{array}}}\Bigl[[{\frak L}_b^{\bar i},{\frak L}_c^{\bar j}], [{\frak L}_{b'}^{\bar k},{\frak L}_{c'}^{\bar l}] \Bigr].$$
By Super Leibniz identity we get
\begin{eqnarray*}
&& \sum\limits_{{\tiny \begin{array}{c}
b,c ,b',c' \in [a]\\
\bar{i},\bar{j},\bar{k},\bar{l} \in \mathbb{Z}_2 \end{array}}} \Bigr[[{\frak L}_b^{\bar i},{\frak L}_c^{\bar j}], [{\frak L}_{b'}^{\bar k},{\frak L}_{c'}^{\bar l}] \Bigl] \\
&& \hspace*{1cm} \subset  \sum\limits_{{\tiny \begin{array}{c}
b,c ,b',c' \in [a]\\
\bar{i},\bar{j},\bar{k},\bar{l} \in \mathbb{Z}_2 \end{array}}} \Bigl( \Bigl[[[{\frak L}_b^{\bar i}, {\frak L}_c^{\bar j}],  {\frak L}_{b'}^{\bar k}],{\frak L}_{c'}^{\bar l}\Bigl] + \Bigr[[[{\frak L}_b^{\bar i},{\frak L}_c^{\bar j}], {\frak L}_{c'}^{\bar l} ], {\frak L}_{b'}^{\bar k}\Bigl]\Bigr)\\
&& \hspace*{1cm} \subset  \sum\limits_{{\tiny \begin{array}{c}
b,c ,b',c' \in [a]\\
\bar{i},\bar{j},\bar{k},\bar{l} \in \mathbb{Z}_2 \end{array}}} \Bigl( \Bigl[[ {\frak L}_\frak{o}^{\bar i+\bar j} ,  {\frak L}_{b'}^{\bar k}],{\frak L}_{c'}^{\bar l}\Bigl] + \Bigr[[ {\frak L}_\frak{o}^{\bar i+\bar j}, {\frak L}_{c'}^{\bar l} ], {\frak L}_{b'}^{\bar k}\Bigl]\Bigr) .
\end{eqnarray*}
 If $[ {\frak L}_\frak{o}^{\bar i+\bar j} ,  {\frak L}_{b'}^{\bar k}] \neq 0$, then $[ {\frak L}_\frak{o}^{\bar i+\bar j} ,  {\frak L}_{b'}^{\bar k}] \subset
{\frak L}_d^{\bar i+\bar j+\bar k}$, with $\frak{o} \star b' = d \in \frak{S} \setminus \{\mathfrak{o}\}$. Then
 $\{d,\tilde{\frak{o}}\}$ is a connection from $d$ to $b'$, so $d \in [a]$. We proceed similarly with $[{\frak L}_\frak{o}^{\bar i+\bar j}, {\frak L}_{c'}^{\bar l}]$  to get
\begin{equation}\label{33}
[{\frak L}_{[a],\frak{o}} ,{\frak L}_{[a],\frak{o}}] \subset
{\frak L}_{[a]}.
\end{equation}
From Equations \eqref{cuatro2}-\eqref{33} we have $[{\frak
L}_{[a]},{\frak L}_{[a]}] \subset {\frak L}_{[a]}$.
\end{proof}

\begin{proposition}\label{pro_prozero}
For any $[a] \neq [t]$ we have $[{\frak
L}_{[a]},{\frak L}_{[t]}] = 0.$
\end{proposition}
\begin{proof}
We have $$[{\frak L}_{[a]}, {\frak L}_{[t]}] = [{\frak L}_{[a],\mathfrak{o}} \oplus V_{[a]}, {\frak L}_{[t],\mathfrak{o}} \oplus V_{[t]}] \subset $$
\begin{equation}\label{cuatro}
\subset [{\frak L}_{[a],\mathfrak{o}}, {\frak L}_{[t],\mathfrak{o}}] + [{\frak L}_{[a],\mathfrak{o}}, V_{[t]}] + [V_{[a]}, {\frak L}_{[t],\mathfrak{o}}]
+ [V_{[a]}, V_{[t]}].
\end{equation}

Consider the above fourth summand $[V_{[a]}, V_{[t]}]$ from (\ref{cuatro}) and suppose there exist $b \in [a]$ and $u \in [t]$ such that $[{\frak L}_b, {\frak L}_u] \neq 0$.   As necessarily $b \star u   \neq \emptyset,$   then
 $\{b, u, \tilde{b}\}$ is a connection from $b$ to $u$. By the transitivity of the connection relation we have $a \sim t$, a contradiction. Hence $[{\frak L}_b, {\frak L}_u] = 0$ and so
\begin{equation}\label{nueve}
[V_{[a]}, V_{[t]}] = 0.
\end{equation}

Consider now the third summand $[V_{[a]}, {\frak L}_{[t],\mathfrak{o}}]$  and suppose exist $b \in [a], u,v \in [t]$ such that  $u \star v = \mathfrak{o}$ and  $[{\frak L}_b, [{\frak L}_u, {\frak
L}_v]] \neq 0$. Hence, there exist $\bar{i}, \bar{j}, \bar{k} \in {\mathbb Z}_2$ such that
$$[{\frak L}_b^{\bar{i}}, [{\frak L}_u^{\bar{j}}, {\frak L}_v^{\bar{k}}]] \neq 0.$$ By Super Leibniz identity we get either $[{\frak L}_b^{\bar{i}}, {\frak L}_u^{\bar{j}}]
\neq 0$ or $[{\frak L}_b^{\bar{i}}, {\frak L}_v^{\bar{k}}] \neq 0$. From here,  in any case $[V_{[a]}, V_{[t]}] \neq 0$, which contradicts Equation (\ref{nueve}).
In a similar way, we show that  the second summand $[{\frak L}_{[a],\mathfrak{o}}, V_{[t]}]$ is zero.

Finally, consider the first summand $[{\frak L}_{[a],\mathfrak{o}}, {\frak L}_{[t],\mathfrak{o}}]$.  Suppose now there exist $b,c \in [a], u,v \in [t]$ such that $b \star c = \mathfrak{o}$ and $u \star v = \mathfrak{o}$,  verifying  $[[{\frak L}_b, {\frak L}_c], [{\frak L}_u, {\frak L}_v]]\neq 0$.
Requiring to the expression with $\mathbb{Z}_2$-graduation we have
$$
\Bigl[  [{\frak L}_b, {\frak L}_c],  [{\frak
L}_{u},{\frak L}_{v}] \Bigr] =  \sum\limits_{{\tiny \begin{array}{c}
\bar{i},\bar{j},\bar{k},\bar{l} \in \mathbb{Z}_2 \end{array}}}\Bigl[[{\frak L}_b^{\bar i},{\frak L}_c^{\bar j}], [{\frak L}_{u}^{\bar k},{\frak L}_{v}^{\bar l}] \Bigr].$$
 Taking now into account Super Leibniz identity we get
\begin{eqnarray*}
&&\hspace*{-1,5cm} \sum\limits_{{\tiny \begin{array}{c}
\bar{i},\bar{j},\bar{k},\bar{l} \in \mathbb{Z}_2 \end{array}}}\Bigl[[
{\frak L}_b^{\bar i},{\frak L}_c^{\bar j}], [{\frak L}_{u}^{\bar k},{\frak L}_{v}^{\bar l}
] \Bigr] \\
&\subset &    \sum\limits_{{\tiny \begin{array}{c}
\bar{i},\bar{j},\bar{k},\bar{l} \in \mathbb{Z}_2 \end{array}}} \Bigl( \Bigl[ [[{\frak L}_b^{\bar i},{\frak L}_c^{\bar j}], {\frak L}_{u}^{\bar k}], {\frak L}_{v}^{\bar l} \Bigr] +  \Bigl[ [ {\frak L}_b^{\bar i},{\frak L}_c^{\bar j} ],{\frak L}_{v}^{\bar l}   ], {\frak L}_{u}^{\bar k} \Bigl]\Bigr)\\
&\subset &    \sum\limits_{{\tiny \begin{array}{c}
\bar{i},\bar{j},\bar{k},\bar{l} \in \mathbb{Z}_2 \end{array}}} \Bigl( \Bigl[ [  {\frak L}_b^{\bar i},[{\frak L}_c^{\bar j} ,{\frak L}_{u}^{\bar k}]], {\frak L}_{v}^{\bar l} \Bigl] +  \Bigl[ [[ {\frak L}_b^{\bar i},{\frak L}_{u}^{\bar k} ],{\frak L}_c^{\bar j} ],{\frak L}_{v}^{\bar l}   ,  \Bigl]\Bigr) \bigcup\\
&& \sum\limits_{{\tiny \begin{array}{c}
\bar{i},\bar{j},\bar{k},\bar{l} \in \mathbb{Z}_2 \end{array}}} \Bigl( \Bigl[ [  {\frak L}_b^{\bar i},[{\frak L}_c^{\bar j} ,{\frak L}_{v}^{\bar l}]],{\frak L}_{u}^{\bar k} \Bigl] + \Bigl[ [ [ {\frak L}_b^{\bar i},{\frak L}_{v}^{\bar l}],{\frak L}_c^{\bar j} ], {\frak L}_{u}^{\bar k} \Bigl]\Bigr)\\.
\end{eqnarray*}
which contradicts (\ref{nueve}), so $[[{\frak L}_b, {\frak L}_c], [{\frak L}_u, {\frak L}_v]]= 0$.

By Equation (\ref{cuatro}) we conclude $[{\frak L}_{[a]},{\frak L}_{[t]}] = 0$ as required.
\end{proof}

Proposition \ref{pro_ideal} asserts that for any $a \in \frak{S} \setminus \{\mathfrak{o}\}, {\frak L}_{[a]}$ is a Leibniz subsuperalgebra of ${\frak L}$ that we call the Leibniz subsuperalgebra of ${\frak
L}$ {\it associated} to $[a]$.

\begin{theorem}\label{teo1}
The following assertions hold.
\begin{enumerate}
\item[{\rm 1.}]  For any $a \in \frak{S} \setminus \{\mathfrak{o}\}$, the Leibniz subsuperalgebra $${\frak L}_{[a]} = {\frak L}_{[a],\mathfrak{o}} \oplus V_{[a]}$$
of ${\frak L}$ associated to $[a]$ is an ideal of ${\frak L}$.
\item[{\rm 2.}] If ${\frak L}$ is simple, then there exists a connection between any two elements of $\frak{S} \setminus \{\mathfrak{o}\}$.
\end{enumerate}
\end{theorem}
\begin{proof}
1. We have
\begin{equation}\label{sumandos}
[{\frak L}_{[a]},{\frak L}] = \Bigl[{\frak L}_{[a],\mathfrak{o}} \oplus V_{[a]}, {\frak L}_{\mathfrak{o}} \oplus \Bigl( \bigoplus\limits_{d \in \frak{S} \setminus \{\mathfrak{o}\}} {\frak L}_d \Bigr)\Bigr].
\end{equation}

In case $[{\frak L}_b, {\frak L}_d] \neq 0$ for some $b \in [a]$ and $d \in \frak{S} \setminus \{\mathfrak{o}\}$, we have that $0 \neq [{\frak L}_b, {\frak L}_d] \subset {\frak L}_n$, with  $n \in \frak{S}$.
 If $n \neq \mathfrak{o}$  the connection $\{b, d\}$ gives us $b \sim n$,
so $n \in [a]$ and then $[{\frak L}_b, {\frak L}_d] \subset V_{[a]}$. If  $n = \mathfrak{o}$ the set $\{\tilde{b}, \mathfrak{o}\}$ is a connection from $b$ to $d$ and so $d \in [a]$, hence $[{\frak L}_b, {\frak L}_d] \subset {\frak L}_{[a],\mathfrak{o}}$. Therefore we get

\begin{equation}\label{eq1}
\Bigl[V_{[a]}, \bigoplus\limits_{d \in \frak{S} \setminus \{\mathfrak{o}\}} {\frak L}_d \Bigr] \subset {\frak L}_{[a]}.
\end{equation}

If we have $[{\frak L}_b , {\frak L}_{\mathfrak{o}}] \neq 0$ for some $b \in [a]$, then $b \star \mathfrak{o} = \{l\}$ with $l \neq \mathfrak{o}$. So $\{b, \mathfrak{o}\}$ is a connection from $b$ to $l$ and then $l \in [a]$. From here

\begin{equation}\label{eq2}
[V_{[a]}, \frak{L}_{\mathfrak{o}}] \subset V_{[a]}.
\end{equation}

Suppose now there exist $b, l \in [a]$ with $b \star l = \{\mathfrak{o}\}$ and $d \in \frak{S} \setminus \{\mathfrak{o}\}$ such that $0 \neq [[{\frak L}_b, {\frak L}_l], {\frak L}_d] \subset {\frak L}_m$, being necessarily $m \neq \mathfrak{o}$.
So for $\bar{i},\bar{j},\bar{k}  \in \mathbb{Z}_2$ such that
$0 \neq [[{\frak L}_b^{\bar i}, {\frak L}_l^{\bar j}], {\frak L}_d^{\bar k}] \subset {\frak L}_m^{\bar i+\bar j+\bar k}$, by applying the Super Leibniz identity we get that either $0 \neq [{\frak L}_b^{\bar i}, [{\frak L}_l^{\bar j}, {\frak L}_d^{\bar k}]] \subset {\frak L}_m^{\bar i+\bar j+\bar k}$ or $0 \neq [[{\frak L}_b^{\bar i},{\frak L}_d^{\bar k}],{\frak L}_l^{\bar j} ] \subset {\frak L}_m^{\bar i+\bar j+\bar k}$.  In the first possibility, there is $n \in \frak{S}$ such that $0 \neq [{\frak L}_l^{\bar j}, {\frak L}_d^{\bar k}] \subset {\frak L}_n^{\bar j+\bar k}$. So the connection $\{b, n\}$  gives us that $b$ is connected to $m$ and so $m \in [a]$. A similar result is obtained in second possibility and so we can assert
\begin{equation}\label{eq3}
\Bigl[{\frak L}_{[a],\mathfrak{o}}, \bigoplus\limits_{d \in \frak{S} \setminus \{\mathfrak{o}\}} {\frak L}_d \Bigr] \subset V_{[a]}.
\end{equation}

Finally, suppose there exist $b, l \in [a]$ with $b \star l = \{\mathfrak{o}\}$ satisfying $0 \neq [[{\frak L}_b, {\frak L}_l], {\frak L}_{\mathfrak{o}}] \subset {\frak L}_n$ with  $n \in \frak{S} $. If $n \neq \mathfrak{o}$ we have as above that $n \in [a]$.   If $n = \mathfrak{o}$, there exist  $\bar{i},\bar{j},\bar{k}  \in \mathbb{Z}_2$   such that $0 \neq [[{\frak L}_b^{\bar i}, {\frak L}_l^{\bar j}], {\frak L}_{\mathfrak{o}}^{\bar k}] \subset {\frak L}_{\mathfrak{o}}^{\bar i+\bar j+\bar k}$.  By the Super Leibniz identity  we have that either $0 \neq [{\frak L}_b^{\bar i}, [{\frak L}_l^{\bar j}, {\frak L}_{\mathfrak{o}}^{\bar k}]] \subset {\frak L}_{\mathfrak{o}}^{\bar i+\bar j+\bar k}$ or $0 \neq [[{\frak L}_b^{\bar i}, {\frak L}_{\mathfrak{o}}^{\bar k}],{\frak L}_l^{\bar j}] \subset {\frak L}_{\mathfrak{o}}^{\bar i+\bar j+\bar k}$ . In the first possibility $l \star \mathfrak{o} = \{r\}$, for some $r \in \mathfrak{S} \setminus \{\mathfrak{o}\}$,   being $\{l, \mathfrak{o}\}$ a connection from $l$ to $r$. From here $r \in [a]$ and $0 \neq [{\frak L}_b, [{\frak L}_l, {\frak L}_{\mathfrak{o}}]] \subset [{\frak L}_b, {\frak L}_r] \subset   {\frak L}_{[a], \mathfrak{o}}$. In the second possibility we obtain the same result, so we can summarize this paragraph by asserting

\begin{equation}\label{eq4}
[{\frak L}_{[a],\mathfrak{o}},{\frak L}_{\mathfrak{o}}] \subset {\frak L}_{[a]}.
\end{equation}

From equations \eqref{sumandos}-\eqref{eq4} we conclude the proof.

2. The simplicity of ${\frak L}$ implies ${\frak L}_{[a]} = {\frak L}$ for any $[a] \in (\frak{S} \setminus \{\mathfrak{o}\}) / \sim$. From here $[a] = \frak{S} \setminus \{\mathfrak{o} \}$ and so any couple of elements in $\frak{S} \setminus \{\mathfrak{o}\}$ is connected.
\end{proof}

\begin{notation}\rm
Let us denote $${\frak L}_{\frak{S},\mathfrak{o}} :=   \Bigl(\sum\limits_{{\tiny \begin{array}{c}
b,c \in  \frak{S} \setminus \{\mathfrak{o}\}\\
b \star c=\{\mathfrak{o}\}     \end{array}}}   [{\frak L}_b, {\frak L}_c] \Bigr) \cap \frak{L}_{\mathfrak{o}}$$
In what follows, we use the terminology $I_{[a]}:=
{\frak L}_{[a]}$ where ${\frak L}_{[a]}$ is one of the ideals of ${\frak L}$ described in Theorem \ref{teo1}-1.
\end{notation}

\begin{theorem} \label{teo2}
For a vector space $\mathcal{U}$ complement  of ${\frak L}_{\frak{S},\mathfrak{o}}$ in ${\frak L}_{\mathfrak{o}}$, it follows $${\frak L} = \mathcal{U} + \sum_{[a] \in  (\frak{S} \setminus \{\mathfrak{o}\}) / \sim}
I_{[a]}.$$
Moreover, $$[I_{[a]},I_{[t]}] = 0,$$ whenever $[a] \neq [t]$.
\end{theorem}
\begin{proof}
From ${\frak L} = {\frak L}_{\mathfrak{o}} \oplus \bigl( \bigoplus\limits_{a \in \frak{S} \setminus \{\mathfrak{o}\}}{\frak L}_a \bigr) = \bigl(\mathcal{U} \oplus
{\frak L}_{\frak{S},\mathfrak{o}} \bigr) \oplus (\bigoplus\limits_{a \in \frak{S} \setminus \{\mathfrak{o}\}} {\frak L}_{a}),$
it follows
$$\bigoplus\limits_{a \in \frak{S} \setminus \{\mathfrak{o}\}}{\frak L}_a =
\bigoplus\limits_{[a] \in  (\frak{S} \setminus \{\mathfrak{o}\}) / \sim} V_{[a]}, \hspace{0.4cm} {\frak L}_{\frak{S},\mathfrak{o}} = \sum_{[a] \in  (\frak{S} \setminus \{\mathfrak{o}\}) / \sim} {\frak L}_{[a],\mathfrak{o}},$$
which implies
$${\frak L} = \bigl(\mathcal{U} \oplus {\frak L}_{\frak{S},\mathfrak{o}} \bigr) \oplus
\bigl(\bigoplus_{a \in \frak{S} \setminus \{\mathfrak{o}\}} {\frak L}_a)= \mathcal{U} + \sum\limits_{[a] \in  (\frak{S} \setminus \{\mathfrak{o}\}) / \sim} I_{[a]},$$ where each $I_{[a]}$ is an ideal of ${\frak L}$ by Theorem \ref{teo1}. Now, given $[a] \neq [t]$, the assertion $$[I_{[a]},I_{[t]}] = 0,$$ follows from Proposition \ref{pro_prozero}.
\end{proof}

\begin{corollary}\label{cooo1}
If $\mathfrak{o}=\emptyset$, then
$${\frak L} = \bigoplus_{[a] \in \frak{S}/\sim} {\frak L}_{[a]},$$ where any ${\frak L}_{[a]}$ is one of the ideals given in Proposition \ref{teo1}.
\end{corollary}

We recall that the {\it center} of a Leibniz superalgebra ${\frak L}$ is the set $${\mathcal Z}({\frak L}) := \Bigl\{x \in {\frak L}:[x, {\frak L}]+[{\frak L},x]=0 \Bigr\},$$ and we say that ${\frak L}_{\mathfrak{o}}$ is {\it tight} whence either ${\frak L}_{\mathfrak{o}} = \{0\}$ or ${\frak L}_{\mathfrak{o}} = \sum_{a,b \in \frak{S} \setminus \{\mathfrak{o}\},a \star b=\{\mathfrak{o}\}}[{\frak L}_a,{\frak L}_b]$.

\begin{corollary}\label{co1}
Suppose ${\frak L}$ is centerless and ${\frak L}_{\mathfrak{o}}$ is tight, then the set-graded Leibniz supe\-ralgebra ${\frak L}$ decomposes as the direct sum of the ideals given in Theorem \ref{teo1},

$${\frak L} = \bigoplus_{[a] \in (\frak{S} \setminus \{\mathfrak{o}\})/\sim} {\frak L}_{[a]}.$$
\end{corollary}

\begin{proof}
By Theorem \ref{teo2}, since $\mathcal{U} = 0,$ we just have to show the direct character of the sum. Given $$x \in \frak{L}_{[a]} \cap \sum\limits_{{\tiny \begin{array}{c}
[t] \in (\frak{S} \setminus \{\mathfrak{o}\})/\sim\\
  t \nsim a  \end{array}}}
  \frak{L}_{[t]},$$ by using the fact $[\frak{L}_{[a]},\frak{L}_{[t]}] = 0$ for $[a] \neq [t]$ we obtain

$$[x,\frak{L}_{[a]}] + \Bigl[x, \sum\limits_{{\tiny \begin{array}{c}
[t] \in (\frak{S} \setminus \{\mathfrak{o}\})/\sim \\
  t \nsim a  \end{array}}} \frak{L}_{[t]} \Bigr] = 0,$$
$$[\frak{L}_{[a]},x] + \Bigl[ \sum\limits_{{\tiny \begin{array}{c}
[t] \in (\frak{S} \setminus \{\mathfrak{o}\})/\sim\\
  t \nsim a  \end{array}}} \frak{L}_{[t]},x \Bigr] = 0.$$

It implies $[x,\frak{L}] = [\frak{L},x] = 0$, that is,   $x \in {\mathcal Z}({\frak L})  $ and so $x  = 0$    as desired.
\end{proof}

\section{The simple components}

In this section we focus on the simplicity of set-graded Leibniz superalgebras by centering our attention in those of maximal length. This terminology is taking patterned from the theory of gradations of Lie and Leibniz algebras (see for example \cite{Alberiogradu, Diamante, Gomez1}). See also \cite{Diamante, Yo1, Yo2, Yo3, Stumme} for examples.

\begin{definition}\rm
We say that a set-graded Leibniz superalgebra ${\frak L}$ is of {\it maximal length} if $\dim ({{\frak L}}_a^{\bar i}) \in \{0, 1\}$ for any $a \in \frak{S} \setminus \{\mathfrak{o}\}$ and  $ \bar{i} \in {\hu Z}_2.$
\end{definition}

Our target is to characterize the simplicity of ${\frak L}$ in terms of connectivity properties in $\frak{S}$. Therefore we would like to attract attention to the definition of simple Leibniz
superalgebra given in Definition \ref{Defsimple}, and the previous discussion.

The following lemma is consequence of the fact that the set of multiplications by elements in ${\frak L}_{\mathfrak{o}}$ is a commuting set of diagonalizable endomorphisms and $I$ is invariant under this set.

\begin{lemma}\label{lema5}
Let ${\frak L} = {\frak L}_{\mathfrak{o}} \oplus (\bigoplus_{a \in \frak{S}\setminus \{\mathfrak{o}\}} {\frak L}_a)$ be a set-graded Leibniz superalgebra.   If $I$ is
 a set-graded ideal   of ${\frak L}$ then $$I = (I \cap {\frak L}_{\mathfrak{o}}) \oplus \bigl(\bigoplus\limits_{a \in \frak{S}\setminus \{\mathfrak{o}\}} (I \cap {\frak L}_a)\bigr).$$
\end{lemma}

From now on ${\frak L} = {\frak L}_{\mathfrak{o}} \oplus \bigl(\bigoplus_{a \in\frak{S} \setminus \{\mathfrak{o}\}}{\frak L}_a \bigr)$ denotes a set-graded Leibniz superalgebra of maximal length without further mention. In this case, we begin by observing that Lemma \ref{lema5} allows us to assert that given any nonzero (set-graded) ideal $I = I^{\bar 0} \oplus I^{\bar 1}$ of ${\frak L}$ then

$$I = (I \cap {\frak L}_{\mathfrak{o}}) \oplus \bigl(\bigoplus\limits_{a \in \frak{S} \setminus \{\mathfrak{o}\}} I \cap {\frak L}_a \bigr)=$$
$$=(I \cap {\frak L}_{\mathfrak{o}}) \oplus \Bigl(\bigoplus\limits_{a \in \frak{S} \setminus \{\mathfrak{o}\}} \bigl((I^{\bar{0}} \cap {\frak
L}_a^{\bar 0}) \oplus (I^{\bar 1} \cap {\frak L}_a^{\bar 1}) \bigr) \Bigr)$$
\begin{equation}\label{max}
=(I \cap {\frak L}_{\mathfrak{o}}) \oplus \bigl(\bigoplus\limits_{a \in \frak{S}_I^{\bar 0}} {\frak L}_a^{\bar 0} \bigr) \oplus \bigl(\bigoplus\limits_{b \in \frak{S}_I^{\bar 1}} {\frak L}_b^{\bar 1} \bigr)
\end{equation}
where $\frak{S}_I^{\bar i} := \bigl\{ a \in \frak{S} \setminus \{\mathfrak{o}\} : I^{\bar i} \cap {\frak L}_a^{\bar i} \neq 0 \bigr\}$, for $\bar{i} \in {\hu Z}_2$. In the important case of the ideal $I = {\frak I}$ defined by \eqref{equi3}, we get
\begin{equation}\label{I}
{\frak I} = ({\frak I} \cap {\frak L}_{\mathfrak{o}}) \oplus \bigl( \bigoplus\limits_{a \in \frak{S}_{{\frak I}}^{\bar 0}} {\frak L}_a^{\bar 0} \bigr) \oplus \bigl(\bigoplus\limits_{b \in \frak{S}_{\frak I}^{\bar 1}} {\frak L}_b^{\bar 1} \bigr)
\end{equation}
with $\frak{S}_{\frak I}^{\bar i} := \bigl\{ a \in \frak{S} \setminus \{\mathfrak{o}\} : {\frak I}^{\bar i} \cap {\frak L}_a^{\bar i} \neq
0 \bigr\} = \bigl\{ a \in \frak{S} \setminus \{\mathfrak{o}\} : 0 \neq {\frak L}_a^{\bar i} \subset {\frak I}^{\bar i} \bigr\},$ for $\bar{i} \in {\hu Z}_2$.
From here, we can write
\begin{equation}\label{separa2}
\frak{S} \setminus \{\mathfrak{o}\} = \underbrace{(\frak{S}_{\frak I}^{\bar 0} \dot{\cup} \frak{S}_{\neg{\frak I}}^{\bar 0})}_{\frak{S}^{\bar 0}\setminus \{\mathfrak{o}\}} \cup \underbrace{(\frak{S}_{\frak I}^{\bar 1} \dot{\cup} \frak{S}_{\neg{\frak I}}^{\bar 1})}_{\frak{S}^{\bar 1}\setminus \{\mathfrak{o}\}}
\end{equation}
where $\hbox{$\frak{S}_{\neg{\frak I}}^{\bar i} := \bigl\{ a \in \frak{S} \setminus \{\mathfrak{o}\} : {\frak L}_a^{\bar i} \neq 0$ and ${\frak I}^{\bar i} \cap {\frak L}_a^{\bar i}  = 0 \bigr\}$}$ for $\bar{i} \in {\hu Z}_2$. We also denote
$$\hbox{$\frak{S}_{\Upsilon} := \frak{S}_{\Upsilon}^{\bar 0} \cup \frak{S}_{\Upsilon}^{\bar 1}$, for $\Upsilon \in \{{\frak I}, \neg {\frak I}\}$.}$$
Hence, we can write
\begin{equation}\label{Llargo}
{\frak L} = \bigl({\frak L}_{\mathfrak{o}}^{\bar 0} \oplus {\frak L}_{\mathfrak{o}}^{\bar 1} \bigr) \oplus
\bigl(\bigoplus\limits_{a \in \frak{S}_{\frak I}^{\bar 0}} {\frak L}_a^{\bar 0} \bigr) \oplus
\bigl(\bigoplus\limits_{b \in \frak{S}_{\neg{\frak I}}^{\bar 0}}  {\frak L}_b^{\bar 0}\bigr) \oplus \bigl(\bigoplus\limits_{c \in \frak{S}_{\frak
I}^{\bar{1}}} {\frak L}_c^{\bar 1} \bigr) \oplus
\bigl(\bigoplus\limits_{d \in \frak{S}_{\neg{\frak I}}^{\bar 1}} {\frak L}_d^{\bar 1} \bigr).
\end{equation}

\begin{remark}\label{re3.1}\rm
Since our aim in this section is to characterize the simplicity of ${\frak L}$, in terms of connections, Theorem \ref{teo1}-2 gives us that we have to center our attention in those set-graded Leibniz superalgebras satisfying ${\frak L}_{\mathfrak{o}} = \sum_{a,b \in \frak{S} \setminus \{\mathfrak{o}\},a \star b=\{\mathfrak{o}\}}[{\frak L}_a,{\frak L}_b]$.   This is for instance the case whence ${\frak L}=[{\frak L},{\frak L}]$.   We would like to note that if ${\frak L}_{\mathfrak{o}} = \sum_{a,b \in \frak{S} \setminus \{\mathfrak{o}\},a \star b = \{\mathfrak{o}\}}[{\frak L}_a,{\frak L}_b]$, then the decomposition given by Equation \eqref{Llargo} and Equation \eqref{equi} show $${\frak L}_{\mathfrak{o}} = \Bigl(\underbrace{\sum\limits_{a \in \frak{S}^{\bar 0},b \in \frak{S}^{\bar 0}_{\neg {\frak I}}, a \star b = \{\mathfrak{o}\}} [{\frak L}_a^{\bar 0}, {\frak L}_b^{\bar 0}] + \sum\limits_{a \in \frak{S}^{\bar 1},b \in \frak{S}^{\bar 1}_{\neg {\frak I}},a \star b = \{\mathfrak{o}\}} [{\frak L}_a^{\bar 1}, {\frak L}_b^{\bar 1}]}_{{\frak L}_{\mathfrak{o}}^{\bar 0}} \Bigr) \oplus$$
\begin{equation}\label{equsimple}
\Bigl(\underbrace{\sum\limits_{a \in \frak{S}^{\bar 1}, b \in \frak{S}^{\bar 0}_{\neg {\frak I}}, a \star b = \{\mathfrak{o}\}} [{\frak L}_a^{\bar 1}, {\frak L}_b^{\bar 0}]
+ \sum\limits_{a \in \frak{S}^{\bar 0}, b \in \frak{S}^{\bar 1}_{\neg {\frak I}}, a \star b = \{\mathfrak{o}\}} [{\frak L}_a^{\bar 0},{\frak L}_b^{\bar 1}]}_{{\frak L}_{\mathfrak{o}}^{\bar{1}}}\Bigr).
\end{equation}
\end{remark}

Now, observe that the concept of connectivity given in Definition \ref{con} is not strong enough to detect if a given $a \in \frak{S}$ belongs to $\frak{S}_{\frak I}^{\bar i}$ or to $\frak{S}_{\neg{{\frak I}}}^{\bar i}$, for some ${\bar i} \in {\mathbb Z}_2$. Consequently we lose the information respect to whether a given component ${\frak L}_a$ intersects to ${\frak I}$ in a non-trivial way or not, which is fundamental to study the simplicity of ${\frak L}$. So, we are
going to make more accurate the previous concept of connection.

\begin{definition}\label{def_cone_S}\rm
Let $a \in \frak{S}_{\Upsilon}^{\bar i}$ and $ b \in \frak{S}_{\Upsilon}^{\bar j}$ with $\Upsilon \in \{{\frak I}, \neg{\frak I}\}$ and ${\bar i}, {\bar j} \in {\mathbb Z}_2$. We say that $a$ is {\it ${\neg{\frak I}}$-connected} to $b,$ denoted by $a \sim_{\neg{\frak I}} b,$ if either $a = b$ or there exists a family of elements $\{r_1,r_2, \dots, r_n\}$  such that for $k=2,\dots,n$ it follows that $r_k \in
\frak{S}_{\neg{\frak I}}^{{\bar i}_k}$ for some ${\bar i}_k \in {\mathbb Z}_2$,  and
\begin{itemize}
\item[1.] $r_1 = a$.
\item[2.] $\phi(\{r_1\},r_2) \in \frak{S}_{\Upsilon}^{{\bar i}+ {\bar i_2}},$\\
$\phi(\phi(\{r_1\},r_2),r_3) \in \frak{S}_{\Upsilon}^{{\bar i}+ {\bar i_2}+{\bar i_3}},$\\
$ \hspace*{1.6cm} \vdots$\\
$\phi(\dots\phi(\{r_1\},r_2),\dots,r_{n-1}) \in \frak{S}_{\Upsilon}^{{\bar i}+ {\bar i_2}+\cdots+{\bar i_{n-1}}},$\\
$\phi(\phi(\dots\phi(\{r_1\},r_2),\dots,r_{n-1}),r_n) \in \frak{S}_{\Upsilon}^{{\bar i}+ {\bar i_2}+\cdots+{\bar i_{n-1}}+ {\bar i_n}}$
\item[3.] $b \in \phi(\phi( \dots \phi(\{r_1\}, r_2), \dots, r_{n-1}), r_n)$ and ${\bar i}+ {\bar i_2}+\cdots+{\bar i_{n-1}}+ {\bar i_n}={\bar j}$.
\end{itemize}
The set $\{r_1,r_2,\dots,r_n\}$ is called a {\it
${\neg{\frak I}}$-connection} from $a$ to $b$.
\end{definition}

Let us introduce the notion of $\frak{S}$-multiplicativity in the framework of set-graded Leibniz superalgebras of maximal length, in a similar way to the ones for split Lie algebras, split Lie superalgebras and split Leibniz algebras among other split algebraic structures (see \cite{Yo1, Yo2, Yo3}  for these notions and examples).

\begin{definition}\label{defmulti}\rm
A set-graded Leibniz superalgebra of maximal length ${\frak L}$ is $\frak{S}$-{\it multiplicative} if the following conditions hold:
\begin{enumerate}
\item Given $a \in \frak{S}_{\neg {\frak I}}^{\bar i}$ and $b \in \frak{S}_{\neg {\frak I}}^{\bar j}$ such that $a \in b \star r$ for some $r \in \frak{S}^{\bar k} \dot{\cup} \tilde{\frak{S}}^{\bar k}$ then $$\frak{L}_a^{\bar i} \subset [\frak{L}_b^{\bar j},\frak{L}_r^{\bar k} + \frak{L}_{\tilde{r}}^{\bar k}].$$
\item Given $c \in \frak{S}_{\frak I}^{\bar i}$ and $d \in \frak{S}_{\frak I}^{\bar j}$ such that $c \in d \star r$ for some $r \in \frak{S}_{\neg \frak I}^{\bar k} \dot{\cup} \tilde{\frak{S}}_{\neg \frak I}^{\bar k}$ then $$\frak{L}_c^{\bar i} \subset [\frak{L}_d^{\bar j},\frak{L}_r^{\bar k} + \frak{L}_{\tilde{r}}^{\bar k}],$$
where $\frak{L}_{l}^{\bar k}$ denotes the empty set when $l$ is a symbol in $\tilde{\frak{S}}$.
\end{enumerate}
\end{definition}

\begin{definition} \rm
We say that $\frak{S}_{\Upsilon}$, with $\Upsilon \in \{{\frak I}, \neg {\frak I}\}$, has all of its elements {\it ${\neg {\frak I}}$-connected} if for any ${\bar i}, {\bar j} \in {\mathbb Z}_2$
we have that $\frak{S}_{\Upsilon}^{\bar i}$ has all of its elements connected to any element in $\frak{S}_{\Upsilon}^{\bar j}$.
\end{definition}

\begin{proposition}\label{nueva1}
Suppose that  ${\frak L}$ is $\frak{S}$-multiplicative, ${\frak L}_{\mathfrak{o}} = \sum_{c,d \in \frak{S} \setminus \{\mathfrak{o}\}, c \star d = \{\mathfrak{o}\}} [{\frak L}_c, {\frak L}_d]$, $|\frak{S}_{\neg {\frak I}}|>1$ and  $\frak{S}_{\neg {\frak I}}$  has all of its elements ${\neg {\frak I}}$-connected. Then any nonzero ideal $I$ of ${\frak L}$ such that $I \nsubseteq {\frak L}_{\mathfrak{o}} + {\frak I}$ satisfies that  $I={\frak L}.$
\end{proposition}

\begin{proof}
Since $I \nsubseteq {\frak L}_{\mathfrak{o}} + {\frak I}$, there exists $a_0 \in \frak{S}_{\neg {\frak I}}^{{\bar i}_0}$ such that
\begin{equation}\label{I2}
0 \neq {\frak L}_{a_0}^{{\bar i}_0} \subset I.
\end{equation}
for some ${\bar i}_0 \in {\hu Z}_2$ (see Equation
(\ref{Llargo})). Given now any $a \in \frak{S}_{\neg {\frak I}}^{\bar j} \setminus \{a_0\}$ with ${\bar j} \in {\hu Z}_2$, being then $0 \neq {\frak L}_a^{\bar j}$, the fact that $a_0$ and $a$ are ${\neg {\frak I}}$-connected gives us a ${\neg {\frak I}}$-connection $\{r_1,r_2,\dots,r_n\} $ from $a_0$ to $a$ such that $$r_1 = a_0 \in \frak{S}^{{\bar
i}_0}_{\neg{\frak I}}, \hspace{0.1cm} r_k \in \frak{S}^{{\bar i}_k}_{\neg{\frak I}} \hspace{0.05cm} {\rm for} \hspace{0.05cm} k=2,\dots,n,$$
$$r_1 \star r_2 \in \frak{S}^{{\bar i_0} + {\bar i_2}}_{\neg {\frak I}}, \dots, r_1 \star r_2 \star \cdots \star r_{n-1} \in \frak{S}^{{\bar i_0} + {\bar i_2} + \cdots + {\bar i_{n-1}}}_{\neg{\frak
I}} \hspace{0.1cm} {\rm and \hspace{0.2cm} finally}$$
$$r_1 \star r_2 \star \dots \star r_{n-1} \star r_n \in \frak{S}^{{\bar i_0} + {\bar i_2} + \cdots + {\bar i_{n-1}} + {\bar i_n}}_{\neg{\frak I}},$$ with $r_1 \star r_2 \star \cdots \star r_{n-1} \star r_n = a$ and ${\bar i_0}+ {\bar i_2} + \cdots + {\bar i_{n-1}} + {\bar i_n} = {\bar j}$.

Consider $r_1, r_2$ and $r_1 \star r_2$. Since
$r_1 = a_0 \in \frak{S}^{{\bar i}_0}_{\neg {\frak I}}$, $r_2 \in \frak{S}^{{\bar i}_2}_{\neg {\frak I}}$ and $r_1 \star r_2 \in \frak{S} ^{{\bar i_0} + {\bar i_2}}_{\neg {\frak I}}$, the $\frak{S}$-multiplicativity and maximal length of ${\frak L}$ show $$0 \neq [{\frak L}_{r_1}^{{\bar
i}_0}, {\frak L}_{r_2}^{{\bar i}_2}] = {\frak
L}_{r_1 \star r_2}^{{\bar i}_0+{\bar i}_2},$$ and by Equation (\ref{I2}) $$0 \neq {\frak L}_{r_1 \star r_2}^{{\bar i}_0+{\bar i}_2} \subset I.$$

We can argue in a similar way from $r_1 \star r_2, r_3$ and $r_1 \star r_2 \star r_3$. That is, $r_1 \star r_2 \in \frak{S}^{{\bar i}_0 + {\bar i}_2}_{\neg {\frak I}}, r_3 \in
\frak{S}^{{\bar i}_3}_{\neg {\frak I}}$ and
$r_1 \star r_2 \star r_3 \in \frak{S}^{{\bar i_0} + {\bar i_2}+ {\bar i_3}}_{\neg {\frak I}}.$ Hence $$0 \neq [{\frak L}_{r_1 \star r_2}^{{\bar i}_0 + {\bar i}_2}, {\frak L}_{r_3}^{{\bar i}_3}] = {\frak L}_{r_1 \star r_2 \star r_3}^{{\bar i}_0+{\bar i}_2 + {\bar i}_3},$$ and by the above $$0 \neq {\frak L}_{r_1 \star r_2 \star r_3}^{{\bar i}_0+{\bar i}_2+{\bar i}_3} \subset I.$$

Following this process with the ${\neg {\frak I}}$-connection $\{r_1,\dots,r_n\}$ we obtain that $$0 \neq{\frak L}_{r_1 \star r_2 \star r_3 \star \cdots \star r_n}^{{\bar i_0} + {\bar i_2} + \cdots + {\bar i_{n-1}} + {\bar i_n}} \subset I$$ and so ${\frak L}_a^{\bar j} \subset I$. In conclusion, for any $a \in \frak{S}_{\neg {\frak I}}^{\bar j} \setminus \{a_0\}$ with ${\bar j} \in {\mathbb Z}_2$, we have
\begin{equation}\label{ootraa}
{\frak L}_a^{\bar j} \subset I.
\end{equation}

Let us now verify that in case $0 \neq {\frak L}_{a_0}^{ \bar i_0 + \bar 1}$ for $  \bar i_0 + \bar 1 \in {\mathbb Z}_2$, we have $0 \neq {\frak L}_{a_0}^{ \bar i_0 + \bar 1} \subset I.$ Indeed, since $|\frak{S}_{\neg {\frak I}}|>1$, we can take $b \in \frak{S}_{\neg \frak I}^{\bar i}$, for some $\bar i \in {\mathbb Z}_2$, such that $b \neq a_0$. By Equation \eqref{ootraa}, it satisfies $0 \neq {\frak L}_b^{\bar i} \subset I$. Hence we can argue as above with the $\frak{S}$-multiplicativity and maximal length of ${\frak L}$ from $b$ instead of $a_0$, to get that in case $a_0 \in \frak{S}_{\neg \frak I}^{\bar k}$ for ${\bar k} \in {\mathbb Z}_2$, then
\begin{equation}\label{lara}
0 \neq {\frak L}_{a_0}^{\bar k} \subset I.
\end{equation}

Since $\frak{L}_{\mathfrak{o}} = \sum_{a,b \in \frak{S}, a \star b = \{\mathfrak{o}\}} [{\frak L}_a, {\frak L}_b]$, Remark \ref{re3.1} and Equations (\ref{ootraa}), (\ref{lara}) give us \begin{equation}\label{lah}
{\frak L}_{\mathfrak{o}} \subset I.
\end{equation}

Now consider $a \in \mathfrak{S}_{\mathfrak I}^{\bar i},$ for any $\bar{i} \in \mathbb{Z}_2$. For all $b \in \mathfrak{S}_{\neg \mathfrak{I}}^{\bar j},$ with $\bar{j} \in \mathbb{Z}_2,$ we have $[\mathfrak{L}_a^{\bar i}, \mathfrak{L}_b^{\bar j}] = \mathfrak{L}_d^{\bar i+\bar j}.$ If $d = \mathfrak{o},$ by \eqref{lah} we get $[\mathfrak{L}_a^{\bar i}, \mathfrak{L}_b^{\bar j}] \subset I.$ Otherwise, $d = a \star b \in \mathfrak{S}_{\mathfrak I}^{\bar i+\bar j}.$ So the $\neg \mathfrak{I}$-connection $\{a,b\}$ implies that  $d \sim_{\neg \mathfrak{I}} a$ and therefore $\mathfrak{L}_a^{\bar i} \subset I$ for any ${\bar i} \in \mathbb{Z}_2.$ So $\mathfrak{I} \subset I.$  The decomposition of ${\frak L}$ in Equation (\ref{Llargo}) finally gives us $I={\frak L}$.
\end{proof}

Let us introduce an interesting notion related to a set-graded Leibniz superalgebra of maximal length ${\frak L}$. We wish to distinguish the  elements of ${\frak L}$ which annihilate the ``Lie type elements'' of ${\frak I}$, so we have the
following definition.

\begin{definition}\label{centerlie}\rm
The {\it Lie-annihilator} of a set-graded Leibniz superalgebra of maximal length ${\frak L}$ is the set $$\hbox{${\mathcal Z}_{Lie}({\frak L}) := \Bigl\{ x \in {\frak L} : [x, {\frak L}_a] + [{\frak
L}_a,x]=0,$ for any $a \notin \frak{S}_{\frak I}\Bigr\}$}.$$
\end{definition}

\noindent Observe that ${\mathcal Z}({\frak L}) \subset {\mathcal Z}_{Lie}({\frak L}).$

\begin{proposition}\label{propoI}
Suppose that  ${\frak L}$ is $\frak{S}$-multiplicative, ${\frak L}_{\mathfrak{o}} = \sum_{b,c \in \frak{S} \setminus \{\mathfrak{o}\}, b \star c = \{\mathfrak{o}\}} [{\frak L}_b, {\frak L}_c]$,  ${\mathcal Z}_{Lie}({\frak L})=0$, $|\frak{S}_{\frak I}|>1$ and $\frak{S}_{\frak I}$ has all of its elements ${\neg {\frak I}}$-connected. Then any nonzero ideal $I$ of ${\frak L}$ such that $I \subset {\frak I}$ satisfies $I = {\frak I}$.
\end{proposition}

\begin{proof}
By Equation \eqref{max} we can write  $$I = (I \cap {\frak L}_{\mathfrak{o}}) \oplus (\bigoplus\limits_{b \in \frak{S}_I^{\bar 0}} {{\frak L}_b^{\bar 0}}) \oplus (\bigoplus\limits_{c \in \frak{S}_I^{\bar 1}} {{\frak L}_c^{\bar 1}})$$ where $\frak{S}^{\bar i}_I := \{d \in \frak{S} : I^{\bar{i}} \cap {\frak L}_d^{\bar i} \neq 0\} = \{d \in \frak{S} : 0 \neq {\frak L}_d^{\bar i} \subset I^{\bar i}\}$ and $\frak{S}_I^{\bar i} \subset \frak{S}_{{\frak I}}^{\bar i}$, for any ${\bar i} \in {\mathbb Z}_2$.

First, we show that \begin{equation}\label{equhache}
{\frak I} \cap {\frak L}_{\mathfrak{o}} \subset {\mathcal Z}_{Lie}({\frak L}).
\end{equation}
\noindent Indeed, fix some ${\bar i} \in {\mathbb Z}_2$ and for any $a \notin \frak{S}_{\frak I}$ and $\bar j \in {\mathbb Z}_2$ we have $$[{\frak I} \cap {\frak L}_{\mathfrak{o}}^{\bar i}, {\frak L}_a^{\bar j}] + [{\frak L}_a^{\bar j}, {\frak I} \cap {\frak L}_{\mathfrak{o}}^{\bar i}] \subset {\frak L}_b^{\bar i + \bar j} \subset {\frak I}$$
where  $b = a \star \mathfrak{o}$.
Let us suppose that  $[{\frak I} \cap {\frak L}_{\mathfrak{o}}^{\bar i}, {\frak L}_a^{\bar j}] + [{\frak L}_a^{\bar j}, {\frak I} \cap {\frak L}_{\mathfrak{o}}^{\bar i}] \neq 0$.
We have $b \in \frak{S}_{\frak I}$.
So $a$ is connected with $b$ using $\{a,\mathfrak{o}\}$ and $a \in \frak{S}_{\frak I}$, a contradiction.
So $[{\frak I} \cap {\frak L}_{\mathfrak{o}}^{\bar i}, {\frak L}_a^{\bar j}] + [{\frak L}_a^{\bar j}, {\frak I} \cap {\frak L}_{\mathfrak{o}}^{\bar i}] = 0$ for each ${\bar i}, {\bar j} \in {\mathbb Z}_2$ and    $a \notin \frak{S}_{\frak I}$. So ${\frak I} \cap {\frak L}_{\mathfrak{o}} \subset {\mathcal Z}_{Lie}({\frak L}).$

From the above ${\frak I} \cap \frak{L}_{\mathfrak{o}} \subset {\mathcal Z}_{Lie}({\frak L})=0$ and we can write
\begin{equation}\label{afinal}
{\frak I} = (\bigoplus\limits_{a \in \frak{S}_{\frak I}^{\bar 0}} {\frak L}_a^{\bar 0}) \oplus (\bigoplus\limits_{b \in \frak{S}_{\frak I}^{\bar 1}} {\frak L}_b^{\bar 1}).
\end{equation}
Taking into account $I \cap \frak{L}_{\mathfrak{o}} \subset {\frak I} \cap \frak{L}_{\mathfrak{o}} = 0$, we also can write
$$I = (\bigoplus\limits_{a \in \frak{S}_I^{\bar 0}} {\frak L}_a^{\bar 0}) \oplus (\bigoplus\limits_{b \in \frak{S}_I^{\bar 1}} {\frak L}_b^{\bar 1}),$$ with  $\frak{S}_I^{\bar i} \subset \frak{S}_{\frak I}^{\bar i}$ for ${\bar i}  \in {\mathbb Z}_2$. Hence, we can take some $a_0 \in \frak{S}_I^{\bar i}$ such that
\begin{equation}\label{betacero}
0 \neq {\frak L}_{a_0}^{\bar i} \subset I.
\end{equation}
Now, we can argue with the $\frak{S}$-multiplicativity and the maximal length of ${\frak L}$ as in Proposition \ref{nueva1}  to conclude that given any $a \in \frak{S}_{\frak I}^{\bar j} \setminus \{a_0\}$, with $\bar j \in {\mathbb Z}_2$, there exists a ${\neg {\frak I}}$-connection
\begin{equation*}\label{cone}
\{r_1,r_2,\dots ,r_n\}
\end{equation*}
from $a_0$ to $a$ such that
\begin{equation}\label{1003}
0 \neq [[\dots[{\frak L}_{a_0}^{\bar i}, {\frak
L}_{r_2}^{{\bar i}_2}], \dots], {\frak L}_{r_n}^{\bar{i}_n}] = {\frak L}_a^{\bar j} \subset I.
\end{equation}
Now we have to study whether ${\frak L}_{a_0}^{\bar k} \subset I$ in case $a_0 \in \frak{S}_{\frak I}^{\bar k}$ for $\bar k \in {\mathbb Z}_2$. To do that, observe that the fact $|\frak{S}_{\frak I}|>1$ allows us to take an element $b \in \frak{S}_{\frak I}^{\bar i} \setminus \{a_0\}$, for some $\bar i \in {\mathbb Z}_2$, such that  satisfies $0 \neq {\frak L}_b^{\bar i} \subset I$, by  Equation \eqref{1003}. Hence we can argue as above with the $\frak{S}$-multiplicativity and maximal length of ${\frak L}$ from $b$ instead of $a_0$, to get that in case $a_0 \in \frak{S}_{\frak I}^{\bar k}$ for ${\bar k} \in {\mathbb Z}_2$, then $0 \neq {\frak L}_{a_0}^{\bar k} \subset I.$
The decomposition of ${\frak I}$ in Equation (\ref{afinal}) finally gives us $I={\frak I}$.
\end{proof}

\begin{theorem}\label{last}
Suppose that  a set-graded Leibniz superalgebra ${\frak L}$  of maximal length is $\frak{S}$-multiplicative,
$|\frak{S}_{\neg{\frak I}}|>1, |\frak{S}_{\frak I}|>1$,
${\frak L}_{\mathfrak{o}} = \sum_{b,c \in \frak{S} \setminus \{\mathfrak{o}\}, b \star c = \{\mathfrak{o}\}} [{\frak L}_b, {\frak L}_c]$ and ${\mathcal Z_{Lie}}(\frak L) =0.$  In this conditions,  ${\frak L}$ is simple if and only if $\frak{S}_{\frak I}$ and $\frak{S}_{\neg {\frak I}}$ have (respectively) all of their elements ${\neg {\frak I}}$-connected.
\end{theorem}

\begin{proof}
First assume that ${\frak L}$ is simple. Fix some $a \in \frak{S}_{\neg {\frak I}}^{\bar i}$ with ${\bar i} \in {\mathbb Z}_2$. Let us denote by $I({\frak L}_a^{\bar i})$ the set-graded ideal of ${\frak L}$ generated by ${\frak L}_a^{\bar i}$. By simplicity $I({\frak
L}_a^{\bar i}) = {\frak L}.$ Observe that Remark
\ref{re3.1} together with Super Leibniz identity allow us to assert that $I({\frak L}_a^{\bar i})$ is contained in the linear span of the set
$$\Bigl\{    [[\dots[v_a, v_{b_1}],\dots],v_{b_n}]; \hspace{0.1cm}
[v_{b_n},[\dots[v_a, v_{b_1}],\dots]]
;$$
$$\hbox{$[[\dots[v_{b_1}, v_a],\dots], v_{b_n}]; \hspace{0.1cm}
[v_{b_n},[\dots[v_{b_1},v_a],\dots]]$ with $0 \neq v_a \in {\frak L}_a^{\bar i}$,}$$
$$\hbox{$0 \neq v_{b_i} \in {\frak L}_{b_i}^{\bar j_i}, b_i \in \frak{S}$, ${\bar j_i} \in {\mathbb Z}_2$ and $n \in {\hu N} \Bigr\}$ }$$ being $a \star b_1, a \star b_1 \star b_2,\dots, a \star b_1 \star b_2 \star \cdots \star b_n$ nonzero elements. From here, given any $a' \in
\frak{S}_{\neg {\frak I}}^{\bar j}$  with ${\bar j} \in {\mathbb Z}_2$,  the above observation gives us that we can write $a' = a \star b_1 \star \cdots \star b_n$ with any $b_i \in \frak{S}_{\neg \frak I}^{\bar j_i}$ and being $a \star b_1 \star \cdots \star b_k \in \frak{S}_{\neg \frak
I}^{\bar i + \bar j_1 + \cdots + \bar j_k}$ (observe that in case some $b_i \in \frak{S}_{\frak I}^{\bar j_i}$ or some ``sum" of elements belongs to $\frak{S}_{\frak I}^{\bar h}$, then either the product involving $v_{b_i}$ or the ``sum" is $\mathfrak{o}$, or implies $a' \in \frak{S}_{\frak I}^{\bar j}$. Hence any $b_i \in \frak{S}_{\neg{\frak I}}^{\bar j_i}$ and the ``sums" are in $\frak{S}_{\neg{\frak I}}^{\bar h}$). From here, we have that $\{a, b_1,\dots, b_n\}$ is a ${\neg{\frak I}}$-connection from $a$ to $a'$ and we can assert that
$\hbox{$\frak{S}_{\neg{\frak I}}$ has all of its elements ${\neg {\frak I}}$-connected.}$
If $\frak{S}_{\frak I} \neq \emptyset$. A similar argument as above implies that  $\frak{S}_{\frak I}$ has all of
its elements ${\neg {\frak I}}$-connected.

Conversely, consider $I$ a nonzero  ideal of ${\frak L}$ and let us show that necessarily either $I = {\frak I}$ or $I = {\frak L}$. Let us distinguish two possibilities:
\begin{itemize}
\item If $I \nsubseteq {\frak L}_{\mathfrak{o}} + {\frak I}$, from Proposition \ref{nueva1}
we obtain that  $I = {\frak L}$.

\item If $I \subset {\frak L}_{\mathfrak{o}} + {\frak I}$, observe that $I \cap {\frak L}_{\mathfrak{o}} \subset {\mathcal Z_{Lie}}(\frak L)$. Indeed, for any $a \notin \frak{S}_{\frak I}$, with $a \neq \mathfrak{o}$ and $\bar i, \bar j \in {\mathbb Z}_2$ we have $[I \cap {\frak L}_{\mathfrak{o}}^{\bar i}, {\frak L}_a^{\bar j}] + [{\frak L}_a^{\bar j}, I \cap {\frak L}_{\mathfrak{o}}^{\bar i}] \subset I \cap
{\frak L}_b^{\bar i + \bar j}  \subset {\frak I}$, where $b = a \star \mathfrak{o} \neq \mathfrak{o}.$ So as above $a$ is connected with $b$, a contradiction with $a \notin \frak{S}_{\frak I}$. So $[I \cap {\frak L}_{\mathfrak{o}}^{\bar i}, {\frak L}_a^{\bar j}] + [{\frak L}_a^{\bar j}, I \cap {\frak L}_{\mathfrak{o}}^{\bar i}]=0$. From here $$I \cap {\frak L}_{\mathfrak{o}} \subset {\mathcal Z_{Lie}}(\frak L)=0.$$
Hence Lemma \ref{lema5} gives us $I \subset {\frak I}$. Finally  Proposition \ref{propoI} implies $I = {\frak I}$ and the proof is complete.
\end{itemize}
\end{proof}
\noindent It remains to study the cases in which either $|\frak{S}_{\neg {\frak I}}| \leq 1$ or $|\frak{S}_{{\frak I}}| \leq 1.$

\begin{proposition}\label{cardinal2}
Suppose that  a set-graded Leibniz superalgebra ${\frak L}$ of maximal length is $\frak{S}$-multiplicative,
${\frak L}_{\mathfrak{o}} = \sum_{b,c \in \frak{S} \setminus \{\mathfrak{o}\}, b \star c = \{\mathfrak{o}\}} [{\frak L}_b, {\frak L}_c]$, ${\mathcal Z_{Lie}}(\frak L) =0$  and $\frak{S}_{\neg {\frak I}}$, $\frak{S}_{\frak I}$  have (respectively) all of their elements ${\neg {\frak I}}$-connected. If either $|\frak{S}_{\neg{\frak I}}| \leq 1$ or $|\frak{S}_{{\frak I}}| \leq 1$, then:

\begin{enumerate}
\item[{\rm (1)}]  either ${\frak L}$ is a simple set-graded Leibniz superalgebra,  

\item[{\rm (2)}] or ${\frak L} = {\frak L}_{\mathfrak{o}} \oplus I \oplus \frak{I},$ where $I = \frak{L}_a$ is a one-dimensional 
subalgebra of ${\frak L} $
with $\frak{S}_{\neg \frak I} = \{a\}$.
\end{enumerate}
\end{proposition}

\begin{proof}
Suppose ${\frak L}$ is not simple. Then there exists a non-zero ideal $I$ of ${\frak L}$ such that $I \notin \{{\frak I}, {\frak L}\}$. Let us distinguish three cases:

1. $|\frak{S}_{\frak I}| \leq 1$ and $|\frak{S}_{\neg {\frak I}}|>1.$ If $I \nsubseteq {\frak L}_{\mathfrak{o}} + {\frak I}$, we have as in Proposition \ref{nueva1} that $I = {\frak L}$. From here $I \subset {\frak L}_{\mathfrak{o}} + {\frak I}.$ Since we get $I \cap {\frak L}_{\mathfrak{o}} = 0$ as in Theorem \ref{last} we have $$I \subset {\frak I}.$$
Let us show that necessarily $|\frak{S}_{{\frak I}}| \neq 1.$ Indeed, if $|\frak{S}_{\frak I}| = 1$, then $\frak{S}_{\frak I} = \{a\}$ and so either $a \in \frak{S}_{\frak I}^{\bar 0} \cap \frak{S}_{\frak I}^{\bar 1}$ or $a \in \frak{S}_{\frak I}^{\bar i}$ with $a \notin \frak{S}_{\frak I}^{\bar i+ \bar 1 }$, for certain  ${\bar i} \in {\mathbb Z}_2$. We have as in Proposition \ref{propoI} that for some  ${\bar j} \in {\mathbb Z}_2$, $0 \neq {\frak L}_a^{\bar j} \subset I$ (in the second case $\bar{j} =\bar{i}$). Since ${\mathcal Z}_{Lie}({\frak L}) = 0$, there exists $b \in \frak{S}^{\bar k}_{\neg {\frak I}}$ such that $[{\frak L}_a^{\bar j}, {\frak L}_b^{\bar k}] + [{\frak L}_b^{\bar k}, {\frak L}_a^{\bar j}] \neq 0$ and so $a \star b \in \frak{S}^{\bar{j}+ \bar{k}}_{\frak I}.$ From here as $\frak{S}_{\frak I} = \{a\}$, $a \star b = a$ and so $b = \mathfrak{o}$, a contradiction. Hence 
the situation $|\frak{S}_{\frak I}| \leq 1$ and $|\frak{S}_{\neg {\frak I}}|>1$
is impossible.

2. $|\frak{S}_{\neg {\frak I}}| \leq 1$ and $|\frak{S}_{\frak I}|>1.$ If $I \subset {\frak L}_{\mathfrak{o}} + {\frak I}$, we have as in item 1. that $I \subset {\frak I}$ and then by Proposition \ref{propoI} we obtain  $I = {\frak I}$, a contradiction because  $I \notin \{{\frak I}, {\frak L}\}$. From here $$I \nsubseteq {\frak L}_{\mathfrak{o}} + {\frak I}.$$ Since Equation \eqref{equsimple} gives us $|\frak{S}_{\neg {\frak I}}| \neq 0$ (in the opposite case ${\frak L}_{\mathfrak{o}} = 0$), hence $|\frak{S}_{\neg {\frak I}}| = 1$. From Equation \eqref{equsimple}  we get straightforward the case (2).

3. $|\frak{S}_{\frak I}| \leq 1$ and $|\frak{S}_{\neg {\frak I}}| \leq 1.$ If $I \subset {\frak L}_{\mathfrak{o}} + {\frak I}$ we have as in item 1.,   a contradiction.
If $I \nsubseteq {\frak L}_{\mathfrak{o}} + {\frak I}$, by arguing as in item 2., we prove that necessarily possibility (2) holds, completing the proof. 
\end{proof}

\end{document}